\documentclass[english,leqno]{elsarticle}

\usepackage{amsthm}
\usepackage{amsmath}
\usepackage{amssymb}
\usepackage{microtype}
\usepackage{enumerate}
\usepackage{hyperref}
\usepackage{todonotes}

\makeatletter
\newtheorem{thm}{Theorem}
\newtheorem{lemma}[thm]{Lemma}

\theoremstyle{definition}

\newsavebox{\fmbox}

\newlength{\dyindent}
{\end{list}}

\newenvironment{dy*}{\begin{list}{}%
{\setlength{\leftmargin}{\dyindent}\setlength{\labelwidth}{\dyindent}%
\addtolength{\labelwidth}{-\labelsep}}%
\item}%
{\end{list}}

\newcommand{\nullity}{\mbox{nullity}}

\begin{document}

\begin{frontmatter}
\title{Two-connected signed graphs with maximum nullity at most two}

\author[gsu]{Marina Arav}

\author[gsu]{Frank J. Hall}

\author[gsu]{Zhongshan Li}

\author[gsu]{Hein van der Holst\corref{cor}\fnref{fn}}

\cortext[cor]{Corresponding author, E-mail: hvanderholst@gsu.edu}

\fntext[fn]{Project sponsored by the National Security Agency under Grant Number H98230-14-1-0152. The United States Government is authorized to reproduce and distribute reprints notwithstanding any copyright notation herein.}

\address[gsu]{Department of Mathematics and Statistics \\
Georgia State University \\
Atlanta, GA 30303, USA
}

\begin{abstract}
A signed graph is a pair $(G,\Sigma)$, where $G=(V,E)$ is a graph (in which parallel edges are permitted, but loops are not) with $V=\{1,\ldots,n\}$ and $\Sigma\subseteq E$. The edges in $\Sigma$ are called odd and the other edges of $E$ even. By $S(G,\Sigma)$ we denote the set of all symmetric $n\times n$ matrices $A=[a_{i,j}]$ with $a_{i,j}<0$ if $i$ and $j$ are adjacent and connected by only even edges, $a_{i,j}>0$ if $i$ and $j$ are adjacent and connected by only odd edges, $a_{i,j}\in \mathbb{R}$ if $i$ and $j$ are connected by both even and odd edges, $a_{i,j}=0$ if $i\not=j$ and $i$ and $j$ are non-adjacent, and $a_{i,i} \in \mathbb{R}$ for all vertices $i$. The parameters $M(G,\Sigma)$ and $\xi(G,\Sigma)$ of a signed graph $(G,\Sigma)$ are the largest nullity of any matrix $A\in S(G,\Sigma)$ and the largest nullity of any matrix $A\in S(G,\Sigma)$  that has the Strong Arnold Property, respectively. In a previous paper, we gave a characterization of signed graphs $(G,\Sigma)$ with $M(G,\Sigma)\leq 1$ and of signed graphs with $\xi(G,\Sigma)\leq 1$. In this paper, we characterize the $2$-connected signed graphs $(G,\Sigma)$ with $M(G,\Sigma)\leq 2$ and the $2$-connected signed graphs $(G,\Sigma)$ with $\xi(G,\Sigma)\leq 2$.
\end{abstract}

\begin{keyword}
signed graph\sep nullity\sep symmetric
\MSC 05C22\sep 05C50\sep 15A03
\end{keyword}
\end{frontmatter}
\section{Introduction}
A \emph{signed graph} is a pair $(G,\Sigma)$, where $G=(V,E)$ is a graph (in which parallel edges are permitted, but loops are not)  and $\Sigma\subseteq E$. (We refer to \cite{Diestel} for the notions and concepts in graph theory.) The edges in $\Sigma$ are called \emph{odd} and the other edges \emph{even}. If $V=\{1,2,\ldots,n\}$, we denote by $S(G,\Sigma)$ the set of all real symmetric $n\times n$ matrices $A=[a_{i,j}]$ with
\begin{itemize}
\item $a_{i,j} < 0$ if $i$ and $j$ are adjacent and all edges between $i$ and $j$ are even, 
\item $a_{i,j}>0$ if $i$ and $j$ are adjacent and all edges between $i$ and $j$ are odd, 
\item $a_{i,j}\in \mathbb{R}$ if $i$ and $j$ are connected by odd and even edges,
\item $a_{i,j}=0$ if $i\not=j$ and $i$ and $j$ are non-adjacent, and 
\item $a_{i,i} \in \mathbb{R}$ for all vertices $i$. 
\end{itemize}
In \cite{AraHalLivdH2013} we introduced for any signed graph $(G,\Sigma)$, among other parameters, the signed graph parameters $M$ and  $\xi$. For a signed graph $(G,\Sigma)$, $M(G,\Sigma)$ is the maximum of the nullities of the matrices in $S(G,\Sigma)$. In order to describe the parameter $\xi$ we need the notion of \emph{Strong Arnold Property} (\emph{SAP} for short). A matrix $A=[a_{i,j}]\in S(G,\Sigma)$ has the SAP if $X=0$ is the only symmetric matrix $X=[x_{i,j}]$ such that $x_{i,j} = 0$ if $i$ and $j$ are adjacent vertices or $i=j$, and $A X = 0$. Then $\xi(G,\Sigma)$ is defined as the largest nullity of any matrix $A\in S(G,\Sigma)$ satisfying the SAP.  It is clear that $\xi(G,\Sigma)\leq M(G,\Sigma)$ for any signed graph $(G,\Sigma)$.
This signed graph parameter $\xi$ is analogous to the parameter $\xi$ for simple graphs introduced by Barioli, Fallat, and Hogben \cite{BFH2005a}.

If $G$ is a graph and $U\subseteq V(G)$, then $\delta(U)$ denotes the set of edges of $G$ with one end in $U$ and the other end in $V(G)-U$. The symmetric difference of two sets $A$ and $B$ is the set $A\Delta B = A\setminus B\cup B\setminus A$.
If $(G,\Sigma)$ is a signed graph and $U\subseteq V(G)$, we say that $(G,\Sigma)$ and $(G,\Sigma\Delta\delta(U))$ are \emph{sign-equivalent} and call the operation $\Sigma\to \Sigma\Delta\delta(U)$ \emph{re-signing on $U$}. Re-signing on $U$ amounts to performing a diagonal similarity on the matrices in $S(G,\Sigma)$, and hence it does not affect $M(G,\Sigma)$ and $\xi(G,\Sigma)$. 

Let $(G,\Sigma)$ be a signed graph.
If $H$ is a subgraph of $G$, then we say that $H$ is \emph{odd} if $\Sigma\cap E(H)$ has an odd number of elements, otherwise  we call $H$ \emph{even}. Zaslavsky showed in \cite{MR676405} that two signed graphs are sign-equivalent if and only if they have the same set of odd cycles. Thus, two signed graphs $(G,\Sigma)$ and $(G,\Sigma')$ that have the same set of odd cycles have $\xi(G,\Sigma) = \xi(G,\Sigma')$. 

In \cite{AraHalLivdH2013}, we showed that a signed graph $(G,\Sigma)$ has $M(G,\Sigma)\leq 1$ if and only if $(G,\Sigma)$ is sign-equivalent to a signed graph $(H,\emptyset)$, where $H$ is a graph whose underlying simple graph is a path. Furthermore, we showed that a signed graph $(G,\Sigma)$ has $\xi(G,\Sigma)\leq 1$ if and only if $(G,\Sigma)$ is sign-equivalent to a signed graph $(H,\emptyset)$, where $H$ is a graph whose underlying simple graph is a disjoint union of paths. Observe that in case the signed graph $(G,\Sigma)$ is connected, $M(G,\Sigma)\leq 1$ if and only if $\xi(G,\Sigma)\leq 1$. In this paper, we characterize the class of $2$-connected signed graphs $(G,\Sigma)$ with $M(G,\Sigma)\leq 2$. We will see that this class coincides with the class of signed graphs $(G,\Sigma)$ with $\xi(G,\Sigma)\leq 2$.

The above characterizations are extensions of results known for simple graphs to signed graphs. 
For a simple graph $G$, denote by $S(G)$  the set of all real symmetric $n\times n$ matrices $A=[a_{i,j}]$ with $a_{i,j}\not=0$ if $i$ and $j$ are connected by an edge, $a_{i,j}=0$ if $i\not=j$ and $i$ and $j$ are non-adjacent, and $a_{i,i} \in \mathbb{R}$ for all vertices $i$. The maximum nullity $M(G)$ of a simple graph $G$ is the maximum of the nullities of the matrices in $S(G)$. Fiedler \cite{MR0244285} proves that a simple graph $G$ has $M(G)\leq 1$ if and only if $G$ is a path. In \cite{MR2549052}, Johnson, Loewy, and Smith characterize the class of simple graphs $G$ with $M(G)\leq 2$. Barioli, Fallat, and Hogben introduced in \cite{BFH2005a} the parameter $\xi$. For a simple graph $G$, $\xi(G)$ is defined as the largest nullity of any matrix $A\in S(G)$ satisfying the SAP. In  \cite{BFH2005a}, they prove that a graph $G$ has $\xi(G)\leq 1$ if and only if $G$ is a subgraph of a path. In \cite{MR2312322}, Hogben and van der Holst give a characterization of the class of simple graphs $G$ with $\xi(G)\leq 2$.

\section{The maximum nullity of some signed graphs}

\emph{Contracting} an edge $e$ with ends $u$ and $v$ in a graph $G$ means deleting $e$ and identifying the vertices $u$ and $v$.
A signed graph $(H,\Gamma)$ is a \emph{weak minor} of a signed graph $(G,\Sigma)$ if $(H,\Gamma)$ can be obtained from $(G,\Sigma)$ by deleting edges and vertices, contracting edges, and re-signing around vertices. We use weak minor to distinguish it from minor in which only even edges are allowed to be contracted (possibly after re-signing around vertices).
The parameter $\xi$ has the nice property that if $(H, \Gamma)$ is a weak minor of the signed graph $(G,\Sigma)$, then $\xi(H,\Gamma)\leq \xi(G,\Sigma)$. 

Let us now introduce some signed graphs. By $K_n^e$ and $K_n^o$ we denote the signed graphs $(K_n, \emptyset)$ and $(K_n,E(K_n))$, respectively. By $K_n^=$ we denote the signed graph $(G,\Sigma)$, where $G$ is the graph obtained from $K_n$ by adding to each edge an edge in parallel, and where $\Sigma$ is the set of edges of $K_n$. By $K_4^i$ we denote the signed graph $(K_4,\{e\})$, where $e$ is an edge of $K_4$. By $K_{2,3}^e$ and $K_{2,3}^i$, we denote the signed graphs $(K_{2,3},\emptyset)$ and $(K_{2,3},\{e\})$, where $e$ is an edge of $K_{2,3}$, respectively.

The following lemma follows from Proposition~4 in \cite{AraHalLivdH2013}.
\begin{lemma}\label{lem:K=}
$M(K_n^=) = \xi(K_n^=) = n$.
\end{lemma}

From Proposition~8 in \cite{AraHalLivdH2013}, the following lemma follows.
\begin{lemma}\label{lem:Kns}
$M(K_n^e) = \xi(K_n^e) = n-1$ and $M(K_n^o) = \xi(K_n^o) = n-1$.
\end{lemma}

The following lemma follows from Proposition~34 in \cite{AraHalLivdH2013}.
\begin{lemma}
$M(K_4^i) = \xi(K_4^i) = 2$.
\end{lemma}

The following lemma follows from Proposition~35 in \cite{AraHalLivdH2013}.
\begin{lemma}\label{lem:K23s}
$M(K_{2,3}^e) = \xi(K_{2,3}^e) = 3$ and $M(K_{2,3}^i) = \xi(K_{2,3}^i) = 2$.
\end{lemma}

From Lemmas~\ref{lem:Kns} and~\ref{lem:K23s}, it follows that signed graphs $(G,\Sigma)$ with $\xi(G,\Sigma)\leq 2$ cannot have a weak $K_4^e$-, $K_4^o$-, or $K_{2,3}^e$-minor.  If a signed graph $(G,\Sigma)$ has no weak $K_4^e$-, $K_4^o$-, or $K_{2,3}^e$-minor, the graph $G$ can still have a $K_4$- or $K_{2,3}$-minor. However, these minors force the signed graph to have additional structure. We will study this in the next section.

By $W_4$ we denote the graph obtained from $C_4$ by adding a new vertex $v$, called the \emph{hub}, and connecting it to each vertex of $C_4$. The subgraph $C_4$ in $W_4$ is called the \emph{rim} of $W_4$. Any edge between $v$ and a vertex of the rim of $W_4$ is called a \emph{spoke}.
Let $e_1,e_2$ be two nonadjacent edges of the $C_4$ in $W_4$.
By $W_4^o$, we denote the signed graph $(W_4, \{e_1,e_2\})$. See Figure~\ref{fig:W4o} for a picture of $W_4^o$; here a bold edge is an odd edge and a thin edge an even edge.
This signed graph appears as a special case in the characterization of $2$-connected signed graphs $(G,\Sigma)$ with $M(G,\Sigma)\leq 2$.

\begin{figure}[h]
	\begin{center}
		\includegraphics[width=0.3\textwidth]{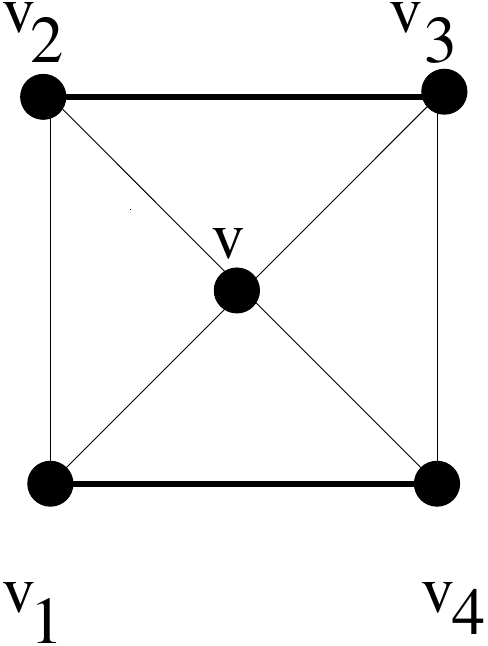}
	\end{center}
	\caption{The signed four-wheel.}\label{fig:W4o}
\end{figure}

\begin{lemma}\label{lem:MW4o}
$M(W_4^o) = \xi(W_4^o) = 2$ 
\end{lemma}
\begin{proof}
Let $v_1,v_2,v_3,v_4$ be the vertices on the rim of $W_4$ in this cyclic order, and let $v$ be the hub of $W_4$. We assume that the edges between $v_1$ and $v_2$ and between $v_3$ and $v_4$ are even. Suppose for a contradiction that $M(W_4^o)\geq 3$. Then there exists a matrix $A=[a_{i,j}]\in S(W_4^o)$ with nullity $\geq 3$. Since $\nullity(A)\geq 3$, there
exist nonzero vectors $x, y\in \ker(A)$ with $x_{v_1} = x_{v_2} = 0$ and $y_{v_1} = y_{v_4} = 0$. If $x_{v} = 0$, then from $A x = 0$ it would follow that $x_{v_4} = x_{v_3} = 0$. This contradiction shows that $x_{v}\not=0$. We may assume that $x_{v} > 0$. Then, since the edge between $v$ and $v_1$ is even, the edge between $v_1$ and $v_4$ is odd, and $A x = 0$, it follows that $x_{v_4} > 0$. Similarly, $x_{v_3} > 0$.  Let $a_{v}$ denote the row of $A$ corresponding to $v$. Then $a_v x = 0$, and, since $x_v > 0$, $x_{v_4} >0$, $x_{v_3} > 0$, it follows that $a_{v,v} > 0$.

We will now do the same with the vector $y$. If $y_v = 0$, then it would follow that $y_{v_2} = y_{v_3} = 0$. This contradiction shows that $y_v \not=0$. We may assume that $y_v>0$. Then, since the edge between $v_1$ and $v_2$ is even, the edge between $v_1$ and $v$ is even, and $A y = 0$, it follows that $y_{v_2} < 0$. Similarly, $y_{v_3} < 0$. Since $a_v y = 0$ and $y_v > 0$, $y_{v_2} < 0$, $y_{v_3} < 0$, it follows that $a_{v,v} < 0$. We have arrived at a contradiction, and we can conclude that $\xi(W_4^o)\leq M(W_4^o) \leq 2$. Since $W_4^o$ contains an odd cycle, $M(W_4^o)\geq \xi(W_4^o) \geq \xi(K_2^=) = 2$, and hence $M(W_4^o) = \xi(W_4^o) = 2$.
\end{proof}

\section{Wide separations}

Let $(G,\Sigma)$ be a signed graph.
A pair $[G_1, G_2]$ of subgraphs of $G$ is a \emph{wide separation} of $(G,\Sigma)$ if there exists an odd $4$-cycle $C_4$ such that $G_1\cup C_4\cup G_2 = G$, $E(G_1)\cap E(C_4)=\emptyset$, $E(G_2)\cap E(C_4)=\emptyset$, $V(G_1)\cap V(G_2)=\emptyset$, $V(G_1)\cap V(C_4)=\{r_1,r_2\}$ and $V(G_2)\cap V(C_4)=\{s_1,s_2\}$, where $r_1$ and $r_2$ are nonadjacent vertices of $C_4$ and $s_1$ and $s_2$ are nonadjacent vertices of $C_4$. We call $r_1,r_2$ the vertices of attachment of $G_1$ and $s_1,s_2$ the vertices of attachment of $G_2$ in the wide separation. See Figure~\ref{fig:widesep} for an illustration. Here a bold edge is an odd edge and a thin edge an even edge.

\begin{figure}[h]
	\begin{center}
		\includegraphics[width=0.6\textwidth]{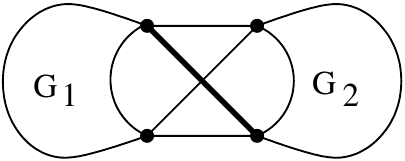}
	\end{center}
	\caption{A wide separation $[G_1, G_2]$.}\label{fig:widesep}
\end{figure}

\begin{lemma}\label{lem:K4d}
Let $(G,\Sigma)$ be a signed graph with no weak minor isomorphic to $K_4^e$, $K_4^o$, or $K_{2,3}^i$.
If $G$ has a $K_4$-minor, but no $W_4$-minor, then $(G,\Sigma)$ has a wide separation.
\end{lemma}
\begin{proof}
Since $G$ has a $K_4$-minor and all vertices of $K_4$ have degree three, $G$ has a subgraph isomorphic to a subdivision of $K_4$. Hence there are distinct vertices $v_1,v_2,v_3,v_4$ and openly disjoint paths $P_1,\ldots,P_6$ of length $\geq 1$ in $G$, where $P_1$ has ends $v_1$ and $v_2$, $P_2$ has ends $v_1$ and $v_3$, $P_3$ has ends $v_1$ and $v_4$, $P_4$ has ends $v_2$ and $v_3$, $P_5$ has ends $v_2$ and $v_4$, and $P_6$ has ends $v_3$ and $v_4$. 
Let $H$ be the subgraph spanned by $v_1,v_2,v_3,v_4$ and $P_1,\ldots,P_6$. Since $(G,\Sigma)$ has no weak $K_4^e$- or $K_4^o$-minor, $(H, \Sigma\cap E(H))$ has no weak $K_4^e$- or $K_4^o$-minor.

By re-signing if needed, we may assume that, in $(G,\Sigma)$, $P_1$ is an odd path and that $P_2,\ldots,P_6$ are even paths. 
We can see that the paths $P_1$, $P_2$, $P_5$, and $P_6$ consist of single edges as follows.
$P_6$ is a single edge, for otherwise, possibly after re-signing, we contract an odd edge in $P_6$ and obtain a signed graph that contains $K_4^o$ as a weak minor. $P_1$ is a single edge, for otherwise, we contract an odd edge in $P_1$ and obtain a signed graph that contains $K_4^e$ as a weak minor. If both $P_3$ and $P_5$ have length $\geq 2$, then, possibly after re-signing, we contract an odd edge in $P_3$ and in $P_5$, and obtain a signed graph that has $K_4^o$ as a weak minor. Hence at least one of $P_3$ and $P_5$ consists of a single edge. In the same way, at least one of $P_2$ and $P_4$ consists of a single edge. If both $P_2$ and $P_3$ have length $\geq 2$, then possibly after re-signing, we contract an odd edge in $P_2$ and in $P_3$, and obtain a signed graph that has $K_4^o$ as a weak minor. Hence at least one of $P_2$ and $P_3$ consists of a single edge. In the same way, at least one of $P_4$ and $P_5$ consists of a single edge. If one of $P_2$ and $P_5$ has length $\geq 2$, then $P_3$ and $P_4$ consist each of a single edge. Hence, we can conclude that both $P_2$ and $P_5$ consist of single edges, or that both $P_3$ and $P_4$ consist of single edges. By symmetry we may assume that $P_2$ and $P_5$ consist of single edges.


Since $G$ has no $W_4$-minor, each path connecting a vertex of $P_3$ to a vertex of $P_4$ must contain at least one vertex of $\{v_1,v_4\}$ and at least one vertex of $\{v_2,v_3\}$.
Thus  $(G,\Sigma)$ has a wide separation.
\end{proof}

The graph with two vertices and no edges is denoted by $K_2^c$.

\begin{lemma}\label{lem:widesepK23o}
Let $(G,\Sigma)$ be a graph with no weak $K_{2,3}^e$-minor. If $G$ has a $K_{2,3}$-minor but no $K_4$-minor, then $(G,\Sigma)$ has a wide separation $[G_1, G_2]$, where $G_1$ is isomorphic to $K_2^c$.
\end{lemma}
\begin{proof}
Since $G$ has a $K_{2,3}$-minor, $G$ has a subgraph $H$ isomorphic to a subdivision of $K_{2,3}$. Hence there are vertices $v_1,v_2,v_3,v_4,v_5$ and openly disjoint paths $P_1,\ldots,P_6$ of length $\geq 1$ in $H$, where $P_1$ has ends $v_1$ and $v_2$, $P_2$ has ends $v_1$ and $v_3$, $P_3$ has ends $v_1$ and $v_4$, $P_4$ has ends $v_2$ and $v_5$, $P_5$ has ends $v_3$ and $v_5$, and $P_6$ has ends $v_4$ and $v_5$. We now view the paths $P_1,\ldots,P_6$ as paths in the signed graph $(H,\Sigma\cap E(H))$. As $(G,\Sigma)$ has no weak $K_{2,3}^e$-minor, $(G,\Sigma)$ has a weak $K_{2,3}^i$-minor. Hence we may re-sign $(H,\Sigma\cap E(H))$ such that $P_1$ is odd and $P_2,\ldots,P_6$ are even.
 
$P_1$ is a single edge, for otherwise, possibly after re-signing, we contract an odd edge in $P_1$ and obtain a signed graph that contains $K_{2,3}^e$ as a weak minor. Similarly, $P_4$ is a single edge. If both $P_2$ and $P_3$ have length $\geq 2$, then, possibly after re-signing, we contract an odd edge in $P_2$ and in $P_3$, and obtain a signed graph that has $K_{2,3}^e$ as a weak minor. Hence, at least one of $P_2$ and $P_3$ consists of a single edge. Similarly, at least one of $P_5$ and $P_6$ consists of a single edge, at least one of $P_2$ and $P_6$ consists of a single edge, and at least one of $P_3$ and $P_5$ consists of a single edge. Hence at most one path of $P_2,P_3,P_5,P_6$ has length $\geq 2$. By symmetry, we may assume that each path of $P_2,P_3,P_5$ consists of a single edge. Let $Q$ be the concatenation of $P_3$ and $P_6$. 
Since $G$ has no $K_4$-minor, there are no paths in $G$ connecting $v_2$ and $v_3$, there are no paths in $G$ connecting $v_2$ and an internal vertex of $Q$, and there are no paths in $G$ connecting $v_3$ and an internal vertex of $Q$. 
Hence, $(G,\Sigma)$ has a wide separation $[G_1,G_2]$, where $G_1$ is isomorphic to $K_2^c$.
\end{proof}

\section{The signed four-wheel}
In this section we show that if $(G,\Sigma)$ is a signed graph with no weak $K_4^e$-, $K_4^o$-, or $K_{2,3}^e$-minor, but $G$ has a $W_4$-minor, then the edges in each parallel class of $(G,\Sigma)$ have the same parity and, after removing all but one edge from each parallel class, we obtain $W_4^o$.

\begin{lemma}\label{lem:W4signequivalent}
Let $(W_4,\Sigma)$ be a signed graph with no weak $K_4^e$-, $K_4^o$-, or $K_{2,3}^e$-minor. Then $(W_4,\Sigma)$ is sign-equivalent to $W_4^o$. 
\end{lemma}
\begin{proof}
If at most one triangle of $(W_4,\Sigma)$ is even, then $(W_4,\Sigma)$ has a weak $K_4^o$-minor.
If at most one triangle of $(W_4,\Sigma)$ is odd, then $(W_,\Sigma)$ has a weak $K_4^e$-minor.
So we may assume that $(W_4,\Sigma)$ has exactly two odd triangles. If they share an edge, then $(W_4,\Sigma)$ has a weak $K_{2,3}^e$-minor. If they do not share an edge, then $(W_4,\Sigma)$ is sign-equivalent to $W_4^o$.
\end{proof}

\begin{lemma}\label{lem:W4oK4oK4e}
Let $(G,\Sigma)$ be obtained from $W_4^o$ by adding an odd or even edge between nonadjacent vertices. Then $(G,\Sigma)$ has a weak $K_4^o$- or $K_4^e$-minor.
\end{lemma}
\begin{proof}
Let $v_1, v_2, v_3, v_4$ be the vertices on the rim of $W_4$ in this order.
Up to symmetry there is only one possibility to add an edge between two nonadjacent vertices in $W_4$. We may assume that we add the edge $e$ between $v_1$ and $v_3$.  If $e$ is an even edge, then the resulting signed graph has a weak $K_4^o$-minor. If $e$ is an odd edge, then the resulting graph has a weak $K_4^e$-minor. 
\end{proof}

%

\begin{lemma}\label{lem:W4ocontract}
Let $(G,\Sigma)$ be a signed graph which has an edge whose contraction yields $W_4^o$. Then $(G,\Sigma)$ has a weak $K_4^o$- or  $K_4^e$-minor.
\end{lemma}
\begin{proof}
If $G$ has a vertex of degree two, then $G$ arises from $W_4$ by inserting a new vertex on an edge $e$ of $W_4$, which results in a path $P$ of length two in $G$. If $e$ is an even edge on the rim of $W_4$, then, possibly after re-signing, we contract all but one odd edge of $P$. The resulting signed graph contains a weak $K_4^o$-minor.
If $e$ is an odd edge on the rim of $W_4$, then, possibly after re-signing, we contract all but one even edge of $P$. The resulting signed graph contains a weak $K_4^e$-minor.
If $e$ is a spoke of $W_4$, then, possibly after re-signing, we contract all but one odd edge of $P$. The resulting signed graph contains a weak $K_4^o$-minor.

So we may assume that $G$ has no vertex of degree two. Then $G$ is isomorphic either to the prism, $C_6^c$ (see Figure~\ref{fig:C6compl}), or to $K_{3,3}$.  
Let $e$ be the edge in $G$ such that contracting $e$ yields $W_4$. 

If $G$ is isomorphic to $K_{3,3}$, then $(G,\Sigma)-e$ contains a weak $K_4^o$-minor.
Suppose next that $G$ is isomorphic to the prism. Then $(G,\Sigma)$ has either two odd triangles or two even triangles.
If $(G,\Sigma)$ has two odd triangles, then, possibly after re-signing, we contract $e$ and another edge not on the two odd triangles, and obtain a signed graph that contains a $K_4^o$ as a signed subgraph. If $(G,\Sigma)$ has two even triangles, then, possibly after re-signing, we contract $e$ and another edge not on the two even triangles, and obtain a signed graph that contains a $K_4^e$ as a signed subgraph.
\end{proof}

\begin{figure}[h]
\begin{center}
\includegraphics[width=0.5\textwidth]{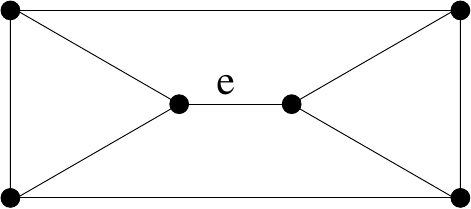}
\end{center}
\caption{The prism with the edge $e$}\label{fig:C6compl}
\end{figure}

\begin{lemma}\label{lem:W4oweak}
Let $(G,\Sigma)$ be a signed graph. If $G$ has a $W_4$-minor,  then at least one of the following holds:
\begin{enumerate}
\item $(G,\Sigma)$ has a weak $K_4^o$-, $K_4^e$-, or $K_{2,3}^e$-minor, or
\item the edges in each parallel class of $(G,\Sigma)$ have the same parity and, after removing from each parallel class all but one edge, we obtain $W_4^o$.
\end{enumerate}
\end{lemma}
\begin{proof}
Suppose for a contradiction that there exists a signed graph $(G,\Sigma)$ such that $G$ has a $W_4$-minor, but $(G,\Sigma)$ has no weak $K_4^o$-, $K_4^e$-, or $K_{2,3}^e$-minor and it is not the case that the edges in each parallel class of $(G,\Sigma)$ have the same parity and, after removing from each parallel class all but one edge, we obtain $W_4^o$. We take $(G,\Sigma)$ with $|V(G)|+|E(G)|$ as small as possible.

If the underlying simple graph of $G$ is isomorphic to $W_4$, then there must be parallel edges of different parity in $(G,\Sigma)$. In this case $(G,\Sigma)$ has a weak $K_4^o$-, $K_4^e$-, or $K_{2,3}^e$-minor by Lemma~\ref{lem:W4signequivalent}. So we may assume that the underlying simple graph of $G$ is not isomorphic to $W_4$. Since $G$ has a $W_4$-minor and $G$ is connected, there exists a signed graph $(H,\Omega)$ such that $(H,\Omega)$ is a minor of $(G,\Sigma)$, the underlying simple graph of $H$ is not isomorphic to $W_4$, and $W_4^o$ can be obtained from $(H,\Omega)$ by deleting or contracting one edge.
If $W_4^o$ can be obtained from $(H,\Omega)$ by deleting an edge $e$, then $e$ connects nonadjacent vertices of $W_4^o$. In this case, $(H,\Omega)$ has a weak $K_4^o$- or $K_4^e$-minor by Lemma~\ref{lem:W4oK4oK4e}, and hence $(G,\Sigma)$ has a $K_4^o$- or $K_4^e$-minor. If $W_4^o$ can be obtained from $(H,\Omega)$ by contracting an edge $e$, then $(H,\Omega)$ has a weak $K_4^o$- or $K_4^e$-minor by Lemma~\ref{lem:W4ocontract}, and hence $(G,\Sigma)$ has  a weak $K_4^o$- or $K_4^e$-minor.
\end{proof}

\section{Partial wide $2$-path}


In this section, we first make some new definitions. A \emph{sided $2$-path} is defined recursively as follows:
\begin{enumerate}
\item Let $T$ be a triangle and let $\mathcal{F}$ be a set of two distinct edges in this triangle. Then $(T,\mathcal{F})$ is a sided $2$-path.
\item If $(G,\mathcal{F})$ is a sided $2$-path and $H$ is obtained from $G$ by adding edges parallel to the edges in $\mathcal{F}$, then $(H,\mathcal{F})$ is a sided $2$-path.
\item Let $(G,\mathcal{F})$ be a sided $2$-path and let $e$ and $f$ be distinct edges in a disjoint triangle $T$. If $H$ is obtained from $G$ by identifying the edge $f$ of $T$ with an edge $h$ in $\mathcal{F}$, then $(H,(\mathcal{F}\setminus \{h\})\cup \{e\})$ is a sided $2$-path.
\end{enumerate}
The edges in $\mathcal{F}$ are called the sides of the sided $2$-path.
A \emph{$2$-path} is a graph $G$ for which there exists a set $\mathcal{F}$ of two distinct edges of $G$ such that $(G,\mathcal{F})$ is a sided $2$-path.
A \emph{partial $2$-path} is a subgraph of a $2$-path. 
A $2$-connected partial $2$-path with no parallel edges is the same as a linear singly edge articulated cycle graph (LSEAC), a type of graph introduced by Johnson et al. \cite{MR2549052}, and it is the same as a linear 2-tree, a type of graph introduced by Hogben and van der Holst \cite{MR2312322}.

\begin{lemma}
Let $G$ be a $2$-connected graph with no $K_4$-, $K_{2,3}$-, or $K_3^=$-minor. Then $G$ is a partial $2$-path.
\end{lemma}
\begin{proof}
Since $G$ has no $K_4$- and no $K_{2,3}$-minor, $G$ is outerplanar. 
Hence $G$ can be embedded in the plane such that all its vertices are incident to the infinite face. Add edges such that all finite faces are either triangles or cycles with exactly two edges, and let the resulting graph be $H$. Construct the following tree $R$ . The vertices of $R$ are all finite faces of the plane embedding. Connect two vertices of the tree if the corresponding faces have an edge in common. Then $R$ is a path. For if not, there would be a face that has edges in common with at least three other faces. Such a graph has a $K_3^=$-minor. By induction it can now be shown that $H$ is a $2$-path. Hence $G$ is a partial $2$-path.
\end{proof}

A pair $\{e,f\}$ of nonadjacent edges in $K_4^i$ is called \emph{split} if both $e$ and $f$ belong to an even and an odd triangle.

A \emph{sided wide $2$-path} $[(G,\Sigma),\mathcal{F}]$ is defined recursively as follows:
\begin{enumerate}
\item Let $(G,\Sigma)$ be an even or odd triangle or a $K_4^i$. If $(G,\Sigma)$ is a triangle, let $\mathcal{F}$ be two distinct edges in this triangle. If $(G,\Sigma) = K_4^i$, let $\mathcal{F}$ be a split pair of edges in $K_4^i$. Then $[(G,\Sigma),\mathcal{F}]$ is a sided wide $2$-path.
\item If $[(G,\Sigma),\mathcal{F}]$ is a sided wide $2$-path and $(H,\Omega)$ is obtained from $(G,\Sigma)$ by adding odd and even edges parallel to edges in $\mathcal{F}$, then $[(H,\Omega),\mathcal{F}]$ is a sided wide $2$-path.
\item Let $[(G,\Sigma),\mathcal{F}]$ be a sided wide $2$-path and let $e$ and $f$ be distinct edges in an even or odd triangle $T$. If $(H,\Omega)$ is obtained from $(G,\Sigma)$ by identifying the edge $f$ of $T$ with an edge $h$ in $\mathcal{F}$, then $[(H,\Omega),(\mathcal{F}\setminus \{h\})\cup \{e\}]$ is a sided wide $2$-path.
\item Let $[(G,\Sigma),\mathcal{F}]$ be a sided wide $2$-path and let $\{e,f\}$ be a split pair of edges in $K_4^i$.  If $(H,\Omega)$ is obtained from $(G,\Sigma)$ by identifying the edge $f$ of $K_4^i$ with an edge $h$ in $\mathcal{F}$, then $[(H,\Omega),(\mathcal{F}\setminus \{h\})\cup \{e\}]$ is a sided wide $2$-path.
\end{enumerate}
The edges in $\mathcal{F}$ are called the sides of the sided wide $2$-path.
A \emph{wide $2$-path} is a signed graph $(G,\Sigma)$ for which there exists a set $\mathcal{F}$ of two distinct edges of $(G,\Sigma)$ such that $[(G,\Sigma),\mathcal{F}]$ is a sided wide $2$-path.
A signed graph $(G,\Sigma)$ is a \emph{partial wide $2$-path} if 
it is a spanning subgraph of a wide $2$-path. Observe that if $G$ is a partial $2$-path, then $(G,\Sigma)$ is a partial wide $2$-path.

\begin{lemma}\label{lem:2connpartwide2path}
Let $(G,\Sigma)$ be a $2$-connected signed graph. 
If $(G,\Sigma)$ has no minor isomorphic to $K_4^e$-, $K_4^o$-, $K_{2,3}^e$, or $K_3^=$, then, after removing in each parallel class all but one edge of the same parity, $(G,\Sigma)$ is either isomorphic to $W_4^o$ or $(G,\Sigma)$ is a partial wide $2$-path.
\end{lemma}
\begin{proof}
If $G$ has a $W_4$-minor, then $(G,\Sigma)$ is isomorphic to $W_4^o$, by Lemma~\ref{lem:W4oweak}. 
We may therefore assume that $G$ has no $W_4$-minor.  

Suppose $G$ has a $K_4$-minor. Then, by Lemma~\ref{lem:K4d}, $(G,\Sigma)$ has a wide separation $[G_1, G_2]$.  
For $i=1,2$, let $(H_i,\Omega_i)$ be obtained from $(G_i,E(G_i)\cap\Sigma)$ by adding between the vertices of attachment of $(G_i,E(G_i)\cap\Sigma)$ in the wide  separation $[G_1, G_2]$ an odd and even edge in parallel. Then $(H_i,\Omega_i)$, $i=1,2$, contains no weak minor isomorphic to $K_4^e$, $K_4^o$, $K_{2,3}^e$, $K_3^=$, or $W_4^o$, for otherwise $(G,\Sigma)$ would contain a weak minor isomorphic to $K_4^e$, $K_4^o$, $K_{2,3}^e$, $K_3^=$, or $W_4^o$. 

Let $u$ and $v$ be the vertices of attachment of $(G_1,E(G_1)\cap\Sigma)$ in the wide separation $[G_1.G_2]$.
Suppose $G_1-\{u,v\}$ contains more than one component; let $C_1,\ldots,C_k$ be the components.  Then, as $G$ is $2$-connected, each $G_1[V(C_i)\cup \{u,v\}]$ contains a path of length $\geq 2$ connecting $u$ and $v$. If there are components $C_j$ and $C_k$ such that both $G_1[V(C_j)\cup \{u,v\}]$ and $G_1[V(C_k)\cup\{u,v\}]$ contain an even path connecting $u$ and $v$, then $(G,\Sigma)$ has a $K_{2,3}^e$-minor, a contradiction. Similarly, there are no two components $C_j$ and $C_k$ such that both $G_1[V(C_j)\cup \{u,v\}]$ and $G_1[V(C_k)\cup\{u,v\}]$ contain an odd path connecting $u$ and $v$. Hence $G_1-\{u,v\}$ has exactly two components $C_1$ and $C_2$. If $G[V(C_1)\cup\{u,v\}]$ contains a path $P$ of length $\geq 3$ connecting $u$ and $v$ and $G[V(C_2)\cup \{u,v\}]$ contains an even path connecting $u$ and $v$, then, possibly after re-signing, we contract an edge of $P$ to obtain an even path. Then $(G,\Sigma)$ has a $K_{2,3}^e$-minor. The cases where $G[V(C_2)\cup \{u,v\}]$ has an odd path connecting $u$ and $v$, and where $G[V(C_2)\cup \{u,v\}]$ has a path of length $\geq 3$ connecting $u$ and $v$ are similar. 

If $G[V(C_1)\cup \{u,v\}]$ has parallel edges whose ends are not $u$ and $v$, then $G[V(C_1)\cup \{u,v\}]$ and $G[V(C_2)\cup \{u,v\}]$ contain paths connecting $u$ and $v$ of equal parity.
Then $(G,\Sigma)$ has a $K_{2,3}^e$. A similar statement holds for $G[V(C_2)\cup \{u,v\}]$. Hence $G[V(C_1)\cup\{u,v\}]$ and $G[V(C_2)\cup \{u,v\}]$ are paths of length $2$ and have different parity. Thus $(H_1,\Omega_1)$ is a subgraph of a sided wide $2$-path where one of the parallel edges between $u$ and $v$ is a side.

Suppose $G_1-\{u,v\}$ contains exactly one component.
By induction $(H_1,\Omega_1)$ is a partial wide $2$-path. Since, in the construction of a wide $2$-path, parallel edges appear only parallel to the edges in $\mathcal{F}$ of a sided wide $2$-path and 
$(H_1,\Omega_1)-\{u,v\}$ has exactly one component, there exists a set $\mathcal{F}_1$ of two distinct edges, one of which is between $u$ and $v$, of $(H_1,\Omega_1)$ such that $[(H_1,\Omega_1),\mathcal{F}_1]$ is a sided wide $2$-path.

Similarly there exists a set $\mathcal{F}_2$ of two distinct edge, one of which is between the vertices of attachments of $(G_2,\Sigma_2)$ such that $[(H_2,\Omega_2),\mathcal{F}_2]$ is a sided wide $2$-path. Then $(G,\Sigma)$ is a partial wide $2$-path. We may therefore assume that $G$ has no $K_4$-minor.

Suppose $G$ has a $K_{2,3}$-minor. Then, by Lemma~\ref{lem:widesepK23o}, $(G,\Sigma)$ has a wide separation $[G_1, G_2]$, where $G_1$ is isomorphic to $K_2^c$. Let $(H_2,\Omega_2)$ be obtained from $(G_2,E(G_2)\cap\Sigma)$ by adding between the vertices attachment of $(G_2,E(G_2)\cap\Sigma)$ in the wide  separation $[G_1, G_2]$ an odd and even edge in parallel. Then $(H_2,\Omega_2)$ contains no minor isomorphic to $K_4^e$, $K_4^o$, $K_{2,3}^e$, $K_3^=$, or $W_4^o$. By induction $(H_2,\Omega_2)$ is a partial wide $2$-path. Similar as above, there is a sided wide $2$-path $[(J_2,\Delta_2),\mathcal{F}_2]$ such that $(H_2,\Omega)$ is a subgraph of $(J_2,\Delta_2)$ and one of the edges of $\mathcal{F}_2$ is an edge between the attachments of 
$(G_2,E(G_2)\cap\Sigma)$ in the wide  separation $[G_1, G_2]$. Then $(G,\Sigma)$ is a partial wide $2$-path. We may therefore assume that $G$ has no $K_{2,3}$-minor. 

Since $G$ has no $K_4$-, $K_{2,3}$-, or $K_3^=$-minor, $G$ is a partial $2$-path, and so $(G,\Sigma)$ is a partial wide $2$-path.
\end{proof}

\begin{lemma}\label{lem:2connwide2pathM2}
If $(G,\Sigma)$ is a $2$-connected partial wide $2$-path, then $M(G,\Sigma)\leq 2$.
\end{lemma}
\begin{proof}
Suppose for a contradiction that $M(G,\Sigma)>2$. Then there exists a matrix $A=[a_{i,j}]\in S(G,\Sigma)$ with $\nullity(A)>2$.  Since $(G,\Sigma)$ is a partial wide $2$-path, $(G,\Sigma)$ is a spanning subgraph of a wide $2$-path $(H,\Omega)$. 

If a wide separation $[H_1,H_2]$ in $(H,\Omega)$ does not yield a wide separation $[G[V(H_1)], G[V(H_2)]]$ in $(G,\Sigma)$, then we replace the $K_4^i$ in $(H,\Omega)$ by two adjacent triangles.
We may therefore assume that for each wide separation $[H_1, H_2]$ in $(H,\Omega)$, $[G[V(H_1)], G[V(H_2)]]$ is a wide separation in $(G,\Sigma)$. 
For a vertex $v$ of $(G,\Sigma)$, we denote by $a_v$ the $v$th row of $A$.

Let $s_1,s_2$ be the two ends of an edge $e$ in the set $\mathcal{F}$ of the wide $2$-path $(H,\Omega)$. As $\nullity(A)>2$, there exists a nonzero vector $x\in \ker(A)$ with $x_{s_1} = x_{s_2} = 0$. If $e$ belongs to a triangle, let $f$ be the edge distinct from $e$ in the construction of the wide $2$-path $(H,\Omega)$. Exactly one end $r_1$ of $f$ belongs to $\{s_1,s_2\}$, let's say $r_1=s_1$, while the other end, $r_2$, is adjacent in $(G,\Sigma)$ to $s_2$. From $a_{s_2} x = 0$ it follows that $x_{r_2} = 0$.

If $e$ belongs to a $K_4^i$, then $e$ belongs to an odd and even triangle of $K_4^i$. Let $r_1$ and $r_2$ be the vertices of this $K_4^i$ distinct from $s_1$ and $s_2$. By symmetry, we may assume that the edges $s_1r_1$, $s_2r_1$, and $s_2r_2$ are even and that the edge $s_1r_2$ is odd. Suppose $x_{r_1} > 0$.
From $a_{s_1} x = 0$, it follows that $x_{r_2} > 0$. From $a_{s_2} x = 0$, it follows that $x_{r_1} < 0$. This contradiction shows that $x_{r_1}\leq 0$. In the same way, it is not possible that $x_{r_1} < 0$. Hence $x_{r_1} = 0$. From $a_{s_1} x = 0$, it then follows that $x_{r_2} = 0$. 

Repeating the above shows that $x=0$, which contradicts that $x$ is a nonzero vector in $\ker(A)$. Thus $M(G,\Sigma)\leq 2$.
\end{proof}

We now arrive at our main result.

\begin{thm}
Let $(G,\Sigma)$ be a $2$-connected signed graph. Then the following are equivalent:
\begin{enumerate}[(i)]
\item\label{item:con2equiv1} $M(G,\Sigma)\leq 2$,
\item\label{item:con2equiv2} $\xi(G,\Sigma)\leq 2$,
\item\label{item:con2equiv3} $(G,\Sigma)$  has no minor isomorphic to $K_3^=$, $K_4^e$, $K_4^o$, or $K_{2,3}^e$.
\item\label{item:con2equiv4} $(G,\Sigma)$ is a partial wide $2$-path or is isomorphic to $W_4^o$.
\end{enumerate}
\end{thm}
\begin{proof}
Since $\xi(G,\Sigma)\leq M(G,\Sigma)$, it is clear that (\ref{item:con2equiv1}) implies (\ref{item:con2equiv2}).

Suppose $(G,\Sigma)$ is a signed graph with $\xi(G,\Sigma)\leq 2$. Since $\xi(K_3^=) = \xi(K_4^e) = \xi(K_4^o) = \xi(K_{2,3}^e) = 3$, $(G,\Sigma)$ has no minor isomorphic to $K_3^=$, $K_4^e$, $K_4^o$, or $K_{2,3}^e$. Hence (\ref{item:con2equiv2}) implies (\ref{item:con2equiv3}).

Suppose the signed graph $(G,\Sigma)$ has no minor isomorphic to $K_3^=$, $K_4^e$, $K_4^o$, or $K_{2,3}^e$. Then, by Lemma~\ref{lem:2connpartwide2path}, $(G,\Sigma)$ is either isomorphic to $W_4^o$ or $(G,\Sigma)$ is a partial wide $2$-path.

If $(G,\Sigma)$ is a partial wide $2$-path, then, by Lemma~\ref{lem:2connwide2pathM2}, $M(G,\Sigma)\leq 2$. Since also $M(W_4^o)\leq 2$, by Lemma~\ref{lem:MW4o},  (\ref{item:con2equiv4}) implies (\ref{item:con2equiv1}).
\end{proof}



\section{References}

\newcommand{\noopsort}[1]{}

\end{document}